\documentclass[12pt]{amsart}
\usepackage{amsmath,amsthm,amssymb}

\newcommand{\al}{\alpha}

\newcommand{\fr}{\mathcal{F}}
\newcommand{\D}{\Delta}
\newcommand{\ity}{\infty}
\newcommand{\C}{\mathbb{C}}

\numberwithin{equation}{section}
\newtheorem{theorem}{Theorem}[section]
\newtheorem{lemma}[theorem]{Lemma}
\newtheorem{corollary}[theorem]{Corollary}
\theoremstyle{remark}

\makeatletter
\@namedef{subjclassname@2010}{%
  \textup{2010} Mathematics Subject Classification}
\makeatother

\begin{document}

\title[A normality Criterion for a Family of Meromorphic Functions ]{A normality Criterion for a Family of Meromorphic Functions}

\thanks{The research work of the first author is supported by research fellowship from UGC India.}

\author[G. Datt]{Gopal Datt}
\address{Department of Mathematics, University of Delhi,
Delhi--110 007, India} \email{ggopal.datt@gmail.com }

\author[S. Kumar]{Sanjay Kumar}

\address{Department of Mathematics, Deen Dayal Upadhyaya College, University of Delhi,
Delhi--110 015, India }
\email{sanjpant@gmail.com}

\begin{abstract}
 Schwick, in \cite{Sch1}, states that let $\fr$ be a family of meromorphic functions on a domain $D$ and if for each   $f\in\fr$, $(f^n)^{(k)}\neq 1$, for $z\in D$, where $n, k$ are positive integers such that $n\geq k+3$,  then 
 $\fr$ is a normal family in $D$. In this paper, we investigate the opposite view that if for each $f\in\fr$, 
 $(f^n)^{(k)}(z)-\psi(z)$ has zeros in $D$, where $\psi(z)$ is a holomorphic function in $D$, 
 then what can be said about the normality of the family $\fr$?
\end{abstract}



 \maketitle

\section{Introduction and main results}
The notion of normal families was introduced by Paul Montel in 1907. Let us begin by recalling the definition: 
A family   of meromorphic functions defined on a domain $D\subset \C$ is said  to be normal in the domain, if 
every sequence in the family  has a subsequence which converges spherically uniformly on compact subsets of  
$D$ to a meromorphic function or to $   \infty$. \\

One important aspect of the theory of complex analytic functions is to find normality criteria for families of 
meromorphic functions. Montel obtained a normality criterion, now known as the fundamental normality test, 
which  says that  {\it a family of meromorphic functions in a domain  is normal if it omits three distinct complex 
numbers.} This  result has undergone  various extensions. In 1975,  Lawrence Zalcman \cite{Zalc 1} proved a 
remarkable  result, now known as Zalcman's Lemma, for  families of meromorphic functions which are not normal 
in a domain. Roughly speaking, it says that {\it a non-normal family can be rescaled at small scale to obtain a 
non-constant meromorphic function in the limit. } This result of Zalcman gave birth to many new normality 
criteria. These normality criteria have been used extensively in complex dynamics for studying the  Julia-Fatou 
dichotomy.\\

Wilhelm Schwick \cite{Sch1}  proved a normality criterion which states that: {\it{Let $n, k$ be positive integers 
such that  $n\geq k+3$, let $\mathcal F$ be a family of functions meromorphic  in $ D$.  If each $f\in \mathcal F$ 
satisfies $(f^n)^{(k)}(z)\neq 1$ for $z\in  D$, then $\mathcal F$ is a normal family.}} This result holds good for 
holomorphic functions in case $n\geq k+1$. Recently  Gerd Dethloff et al. \cite{dethloff} came up with new 
normality criteria, which improved the result given by Schwick \cite{Sch1}.

\begin{theorem}\label{thm1}
Let  $a_1, a_2,\ldots, a_q,$ be $q$ distinct non-zero complex values and $l_1, l_2,\ldots, l_q$ be $q$ positive 
integers $($or $+\ity$$)$, where $q\geq1$.  Let $n$ be a non-negative integer, and  
$n_1,\ldots, n_k, t_1,\ldots, t_k$  positive integers $(k\geq1).$ Let $\fr$ be a family of 
meromorphic functions in a  domain $D$ such that for every 
$f \in \fr$,  all zeros of $f^n(f^{n_1})^{(t_1)}\ldots(f^{n_k})^{(t_k)}-a_i$ are of multiplicity 
at least $l_i$, for $i=1, 2, \ldots, q$. Assume that
\begin{enumerate}
\item[(a)] $n_j\geq t_j$ for all $1\leq j\leq k,$ and $l_i\geq2$ for all $1\leq i\leq q.$
\item[(b)] $\displaystyle{\sum_{i=1}^{q}\frac{1}{l_i}<\frac{qn-2+
\sum_{j=1}^{k}q(n_j-t_j)}{n+\sum_{j=1}^{k}(n_j+t_j)}}$.
\end{enumerate}
Then $\fr$ is normal in $D$.
\end{theorem}

 For the case of holomorphic functions they proved the following strengthened version:
\begin{theorem}\label{thm1}
Let  $a_1, a_2,\ldots, a_q,$ be $q$ distinct non-zero complex values and $l_1, l_2,\ldots, l_q$ be $q$ positive 
integers $($or $+\ity$$)$, where $q\geq1$.  Let $n$ be a non-negative integer,  $n_1,\ldots, n_k, t_1,\ldots, t_k$ 
positive integers $(k\geq1).$ Let $\fr$ be a family of holomorphic functions in a  domain $D$ such that for every 
$f \in \fr$,  all zeros of $f^n(f^{n_1})^{(t_1)}\ldots(f^{n_k})^{(t_k)}-a_i$ are of multiplicity at least $l_i$, for 
$i=1, 2, \ldots, q$. Assume that
\begin{enumerate}
\item[(a)] ${n_j\geq t_j}$ for all $1\leq j\leq k,$ and $l_i\geq2$ for all $1\leq i\leq q.$
\item[(b)] $\displaystyle{\sum_{i=1}^{q}\frac{1}{l_i}<
\frac{qn-1+\sum_{j=1}^{k}q(n_j-t_j)}{n+\sum_{j=1}^{k}(n_j)}}$.
\end{enumerate}
Then $\fr$ is normal in $D$.
\end{theorem}

 The main aim of  this paper is to obtain a normality criterion with the condition $(f^n)^{(k)}(z)-\psi(z)$ has  
 zeros with multiplicities at least $m\geq 1$, where $\psi(z)(\not\equiv 0)$ is a holomorphic function.
\begin{theorem}\label{theorem 1}
Let $\fr$ be a family of meromorphic functions on a domain $D\subset \C$ and let $k, p, m$ and $n$ be positive 
integers satisfying 
\begin{enumerate}
\item[(a)] $\displaystyle{\frac{k+2(p+1)}{n}+\frac{p+1}{m} +\frac{k(p+1)}{mn}<1}$,
\item[(b)] $\displaystyle{\frac{p+1}{m}+2\left(\frac{k+p+1}{n}\right)<1}$.
\end{enumerate}
Let $\psi(z)(\not\equiv 0)$ be a holomorphic function in $D$, which has zeros of multiplicity at most $p$. 
Suppose that, for every function $f\in\fr$,
\begin{enumerate}
\item $(f^n)^{(k)}(z)-\psi(z)$ has zeros of multiplicity at least $m$,
\item $\psi(z)$ and $f(z)$ have no common zeros in $D$,
\item number of poles of $f$ $($if they exist$)$ are greater than or equal to the number of zeros of $f$.
\end{enumerate}
Then $\fr$ is normal in $D$.
\end{theorem}
As an application of our result the following corollary is strengthened version of Schwick's result \cite{Sch1}.
\begin{corollary}
Let $n, k$ be positive integers such that  $n\geq k+2$, let $\mathcal F$ be a family of functions meromorphic  in 
$ D$.  If each $f\in \mathcal F$ satisfies $(f^n)^{(k)}(z)\neq 1$ for $z\in  D$, 
then $\mathcal F$ is a normal family.
\end{corollary}

\section{Some Notations}
Let $\Delta=\{z: |z|<1\}$ be the unit disk and $\Delta(z_0, r):=\{z: |z-z_0|<r\}$ and $\Delta'=\{z: 0<|z|<1\}$. 
We use the following standard functions of value distribution theory, namely
\begin{center}
$T(r,f),  m(r,f),  N(r,f)\ \text{and}\ \overline{N}(r,f)$.
\end{center}
We let $S(r,f)$ be any function satisfying
\begin{center}
$S(r,f)=o\big(T(r,f)\big)$,  as $r\rightarrow +\ity,$
\end{center}
 possibly outside of a set with finite measure.
 
 \section{Preliminary results}
  In order to prove our results we need  the following Lemmas.
\begin{lemma}[Zalcman's lemma]\cite{Zalc 1,Zalc}\label{lem1} Let $\mathcal F$ be a family of meromorphic  
functions in the unit disk  $\Delta$, with the property that for every function $f\in \mathcal F,$  the zeros of $f$ are of multiplicity at least $l$ and the poles of $f$ are of multiplicity at least $k$ . If $\mathcal F$ is not normal at $z_0$ in $\Delta$, then for  $-l< \alpha <k$, there exist
\begin{enumerate}
\item{ a sequence of complex numbers $z_n \rightarrow z_0$, $|z_n|<r<1$},
\item{ a sequence of functions $f_n\in \mathcal F$ },
\item{ a sequence of positive numbers $\rho_n \rightarrow 0$},
\end{enumerate}
such that $g_n(\zeta)=\rho_n^{\alpha}f_n(z_n+\rho_n\zeta) $ converges to a non-constant meromorphic function $g$ on $\C$ with $g^{\#}(\zeta)\leq g^{\#}(0)=1$. Moreover $g$ is of order at most two .
\end{lemma}

\begin{lemma}\label{lemma 5}\cite{lahiri,lahiri 1}
Let ${R={P}/{Q}} $ be a rational function and $Q$ be non-constant. Then $\left(R^{(k)}\right)_{\ity}\leq(R)_{\ity}-k,$ where $k$ is a positive integer,  $(R)_{\ity}=$ $\text{deg}\ (P)-\text{deg}\ (Q)$ and $\text{deg}\ (P)$ denotes the degree of P.
\end{lemma}

\begin{lemma}\label{milloux}\cite{Hay,ccyang,Yang} 
Suppose $f(z)$ is a non-constant meromorphic function in the complex plane and $k$ is a positive integer. Then
\begin{equation}
T(r, f)< \overline{N}(r,f) + N\left(r, \frac{1}{f} \right) + N\left(r, \frac{1}{f^{(k)}-1} \right) -N\left(r, \frac{1}{f^{(k+1)}} \right) + S(r, f).
\end{equation}
\end{lemma}

\begin{lemma}\label{lema 4}
Let $\fr = \{f_j\}$ be a family of meromorphic functions defined on $D\subset \C$. Let $\phi_j(z)$ be a sequence 
of holomorphic functions on $D$ such that $\phi_j(z)\rightarrow \phi(z)$ locally uniformly on $D$, where $\phi(z)
\neq 0$ is a holomorphic function on $D$.  Let $k, m \text{\ and\ } n$ be three positive integers such that $$\displaystyle{\frac{k+2}{n}+\frac{1}{m}+\frac{k}{mn}<1}.$$ 
If each zero of $\left(f_j^n \right)^{(k)}-\phi_j(z)$ has multiplicity at least $m$, then $\fr$ is normal in $D$.
\end{lemma}
\begin{proof}
Since normality is a local property, we assume that
${D}=\Delta$. Suppose that $\fr$ is not normal
in $\Delta$. Then there exists at least one point $z_0$ such that
$\fr$ is not normal  at the point $z_0$ in $\Delta$. Without loss of
generality we assume that $z_0=0$. Then by Lemma \ref{lem1}, for $\al=k$  there
exist
\begin{enumerate}
\item{ a sequence of complex numbers $z_j \rightarrow 0$, $|z_j|<r<1$},
\item{ a sequence of functions $f_j\in \mathcal F$ },
\item{ a sequence of positive numbers $\rho_j \rightarrow 0$},
\end{enumerate}
such that $\displaystyle{g_j(\zeta)=\frac{f_j^n(z_j+\rho_j\zeta)}{\rho_j^{k}}} $ converges locally uniformly  to a 
non-constant meromorphic function $g(\zeta)$ on $\C$ with $g^{\#}(\zeta)\leq g^{\#}(0)=1$. By Hurwitz's 
Theorem each zero of $g(\zeta)$ is of multiplicity at least $n$. Thus we have
\begin{align}
g_j^{(k)}(\zeta) - \phi_j(z_j + \rho_j\zeta) &= (f_j^n)^{(k)}(z_j+\rho_j\zeta)-\phi_j(z_j+\rho_j\zeta)\notag\\
&\rightarrow g^{(k)}(\zeta) - \phi(z_0).\label{eq3.1}
\end{align}
Since $(f_j^n)^{(k)}(z_j+\rho_j\zeta)-\phi_j(z_j+\rho_j\zeta)$ has zeros only of multiplicities at least $m$, 
therefore, by Hurwitz's Theorem, $g^{(k)}(\zeta)-\phi(z_0)$ has zeros only of multiplicities  at least $m.$ Now, 
by Lemma \ref{milloux} and Nevanlinna's first fundamental theorem, we get
\begin{align*}
T(r,g)&< \overline{N}(r, g)+ N\left(r, \frac{1}{g} \right)+ N\left(r, \frac{1}{g^{(k)}-\phi(z_0)} \right) -N\left(r, \frac{1}{g^{(k+1)}} \right) + S(r, g)\\
&\leq \frac{1}{n}N(r, g) + (k+1)\overline{N}\left(r, \frac{1}{g} \right)+  \overline{N}\left(r, \frac{1}{g^{(k)}-\phi(z_0)}\right)+ S(r, g)\\
&\leq\frac{1}{n}T(r, g) + \frac{k+1}{n}T(r, g) + \frac{1}{m}N\left(r, \frac{1}{g^{(k)}-\phi(z_0)}\right) + S(r, g)\\
&\leq \frac{k+2}{n}T(r, g) + \frac{1}{m}\left( T(r, g) + k\overline{N}(r, g) \right)+ S(r, g)\\
&\leq \left(\frac{k+2}{n} +\frac{1}{m} +\frac{k}{mn} \right) T(r, g) + S(r, g).
\end{align*}
 This is a contradiction since $\displaystyle{\frac{k+2}{n} +\frac{1}{m} +\frac{k}{mn}<1}$ and $g(\zeta)$ is a non-constant meromorphic function. This proves the Lemma.
\end{proof}

\section{Proof of Theorem \ref{theorem 1}}

Here we prove slightly stronger version of Theorem \ref{theorem 1}
\begin{theorem}
Let $\fr=\{f_j\}$ be a sequence of meromorphic functions on a domain $D\in \C$, let $k, p, m$ and $n$ be positive integers satisfying 
\begin{enumerate}
\item[(a)] $\displaystyle{\frac{k+2(p+1)}{n}+\frac{p+1}{m} +\frac{k(p+1)}{mn}<1}$,
\item[(b)] $\displaystyle{\frac{p+1}{m}+2\left(\frac{k+p+1}{n}\right)<1}$.
\end{enumerate}
Let $\{\psi_j(z)\}$ be a sequence of holomorphic functions having zeros of multiplicity at most $p$ on $D$ 
such that $\psi_j(z)\rightarrow \psi(z)$, locally uniformly  on $D$, where $\psi(z)(\not\equiv 0)$ is a holomorphic 
function on $D$.  Suppose that, for every function $f\in\fr$,
\begin{enumerate}
\item $(f_j^n)^{(k)}(z)-\psi_j(z)$ has zeros only of multiplicity at least $m$,
\item $\psi_j(z)$ and $f_j(z)$ have no common zeros in $D$,
\item number of poles of $f_j$ $($if they exist$)$ are greater than or equal to the number of zeros of $f_j$.
\end{enumerate}
Then $\fr$ is normal in $D$.

\end{theorem}
\begin{proof}Since normality is a local property, we assume that
${D}=\Delta$ and $\psi_j(z)=z^l\phi_j(z)$, where $l$ is a non-negative integer with $l\leq p$, $\phi_j(0)\rightarrow 1$ and $\psi_j(z)\neq 0$ on $\Delta'$. By Lemma \ref{lema 4} it is sufficient  to prove that $\mathcal F$ is normal at $z=0.$ Consider the family $$\mathcal G = \left\{g_j(z)=\frac{f_j^n(z)}{\psi_j(z)}: f_j \in \mathcal F, z \in \Delta\right\}.$$ Since $\psi_j(z)$ and $f_j(z)$ have no common zeros for each $f_j\in\fr$ we get $g_j(0)=\ity \ \forall \ g_j\in\mathcal G$ and $g_j$ has a pole of order at least $l$ at $0$. First we prove that $\mathcal G$ is normal in $\D$. Our  proof proceeds by contradiction. Suppose that $\mathcal G$ is not normal at $z_0\in \D$. By Lemma \ref{lem1}, there exist sequences $g_j\in \mathcal G, z_j\rightarrow z_0$ and $\rho_j\rightarrow 0^+$ such that $$ G_j(\zeta)=\frac{g_j(z_j+\rho_j\zeta)}{\rho_j^k}=\frac{f_j^n(z_j+\rho_j\zeta)}{\rho_j^k\psi_j(z_j+\rho_j\zeta)}\rightarrow G(\zeta)$$ locally uniformly with respect to the spherical metric, where $G(\zeta)$ is a non-constant meromorphic function on $\C$, all of whose zeros are of multiplicity at least $n$. There are  two cases to consider.\\

\underline{Case 1.} Suppose $\displaystyle{\frac{z_j}{\rho_j}\rightarrow \ity}.$ Then since $\displaystyle{G_j\left(\frac{z_j}{\rho_j}\right)=\frac{g_j(0)}{\rho_j^k}}$, the pole of $G_j$ corresponding to that of $g_j$ drifts off to infinity and $G(\zeta)$ has only multiple poles. We claim that $$\frac{\left(f_j^n\right)^{(k)}(z_j+\rho_j\zeta)}{\psi_j(z_j+\rho_j\zeta)}\rightarrow G^{(k)}(\zeta)$$ uniformly on compact subsets of $\C$ disjoint from the poles of $G$. Indeed
\begin{align*}
g_j^{(k)}(z)=\frac{\left(f_j^n \right)^{(k)}(z)}{\psi_j(z)}-\begin{pmatrix} k\\1  \end{pmatrix}g_j^{(k-1)}(z)\frac{\psi_j'(z)}{\psi_j(z)}-\begin{pmatrix} k\\ 2 \end{pmatrix}g_j^{(k-2)}(z)\frac{\psi_j''(z)}{\psi_j(z)}\ldots - g_j(z)\frac{\psi_j^{(k)}(z)}{\psi_j(z)}
\end{align*}
and
\begin{align*}
G_j^{(k)}(\zeta)&= g_j^{(k)}(z_j+\rho_j\zeta)\\
&=\frac{\left(f_j^n \right)^{(k)}(z_j+\rho_j\zeta)}{\psi_j(z_j+\rho_j\zeta)}-\sum_{i=1}^{k}\begin{pmatrix} k\\i  \end{pmatrix}g_j^{(k-i)}(z_j+\rho_j\zeta)\frac{\psi_j^{(i)}(z_j+\rho_j\zeta)}{\psi_j(z_j+\rho_j\zeta)}\\
&=\frac{\left(f_j^n \right)^{(k)}(z_j+\rho_j\zeta)}{\psi_j(z_j+\rho_j\zeta)}-\begin{pmatrix} k\\1  \end{pmatrix}g_j^{(k-1)}(z_j+\rho_j\zeta)\left(\frac{l}{z_j+\rho_j\zeta}+\frac{\phi_j'(z_j+\rho_j\zeta)}{\phi_j(z_j+\rho_j\zeta)} \right)\ldots\\
&  \qquad g_j(z_j+\rho_j\zeta)\left(\sum_{i=0}^{k}\frac{l!}{(l-k+i)!(z_J+\rho_j\zeta)^{k-i}}\frac{\phi_j^{(i)}(z_j+\rho_j\zeta)}{\phi_j(z_j+\rho_j\zeta)} \right)\\
& =\frac{\left(f_j^n \right)^{(k)}(z_j+\rho_j\zeta)}{\psi_j(z_j+\rho_j\zeta)}-\begin{pmatrix} k\\1  \end{pmatrix}\frac{g_j^{(k-1)}(z_j+\rho_j\zeta)}{\rho_j}\left(\frac{l\rho_j}{z_j+\rho_j\zeta}+\frac{\rho_j\phi_j'(z_j+\rho_j\zeta)}{\phi_j(z_j+\rho_j\zeta)} \right)\ldots\\
&  \qquad \frac{g_j(z_j+\rho_j\zeta)}{\rho_j^k}\left(\sum_{i=0}^{k}\frac{l!\rho_j^{k-i}}{(l-k+i)!(z_J+\rho_j\zeta)^{k-i}}\frac{\rho_j^i\phi_j^{(i)}(z_j+\rho_j\zeta)}{\phi_j(z_j+\rho_j\zeta)} \right).
\end{align*}
Also, we have $\displaystyle\lim_{j\rightarrow \ity}\left(\frac{\rho_j}{z_j+\rho_j\zeta}\right)=0$ and $\displaystyle{\lim_{j\rightarrow\ity}\rho_j^i\frac{\phi_j^{(i)}(z_j+\rho_j\zeta)}{\phi_j(z_j+\rho_j\zeta)}}=0$ \ $(i=1, 2,\ldots , k)$.\\

Note that $\displaystyle{\frac{ g_j^{(k-i)}(z_j+\rho_j\zeta)}{\rho_j^i}}$ is locally bounded on $\C$ minus the set of poles of $G(\zeta)$ since $\displaystyle{\frac{g_j(z_j+\rho_j\zeta)}{\rho_j^k}\rightarrow G(\zeta)}.$ Therefore, on every compact subset of $\C$ which does not contain any pole of $G(\zeta)$, we have $$\frac{\left(f_j^n\right)^{(k)}(z_j+\rho_j\zeta)}{\psi_j(z_j+\rho_j\zeta)}\rightarrow G^{(k)}(\zeta).$$ By Hurwitz's Theorem $(f_j^n)^{(k)}(z_j+\rho_j\zeta)-\psi_j(z_j+\rho_j\zeta)$ has zeros only of multiplicity at least $m$ and $\psi_j(z_j+\rho_j\zeta)$ has zeros only at $\displaystyle{\zeta=-\frac{z_j}{\rho_j}\rightarrow\ity}.$ Therefore, we have $G^{(k)}(\zeta)-1$ has zeros only of multiplicity at least $m$. Now, by Milloux's inequality and Nevanlinna's first fundamental theorem, we have
\begin{align*}
T(r,G)&\leq \overline{N}(r,G) + N\left(r, \frac{1}{G}\right) + N\left(r, \frac{1}{G^{(k)}-1}\right) - N\left(r, \frac{1}{G^{(k+1)}} \right) + S(r, G)\\
&\leq \frac{N\left(r, G\right)}{n} + (k+1)\overline{N}\left(r, \frac{1}{G} \right) + \overline{N}\left(r, \frac{1}{G^{(k)}-1} \right) + S(r, G)\\
&\leq \frac{N\left(r, G\right)}{n} + \frac{k+1}{n}N\left(r, \frac{1}{G} \right) + \frac{1}{m}N\left(r, \frac{1}{G^{(k)}-1} \right) + S(r, G)\\
&\leq \frac{k+2}{n}T(r, G) +\frac{1}{m}\left(T(r,G) + k\overline{N}(r, G) \right) + S(r, G)\\
&\leq \left(\frac{k+2}{n} + \frac{1}{m} + \frac{k}{nm}\right)T(r,G) +S(r, G).
\end{align*}
This contradicts that $G(\zeta)$ is a non-constant meromorphic function on $\C$.\\

\underline{Case 2.} $\displaystyle{\frac{z_j}{\rho_j}\not\rightarrow \ity}, $ taking  a subsequence and renumbering, we may assume that $\displaystyle{\frac{z_j}{\rho_j}\rightarrow\alpha},$ a finite complex number. Then $$\displaystyle{\frac{g_j(\rho_j\zeta)}{\rho_j^k}=\frac{g_j\left(z_j+ \rho_j\left(\frac{\zeta-z_j}{\rho_j}\right)\right)}{\rho_j^k}=\frac{G_j\left(\zeta-\frac{z_j}{\rho_j}\right)}{\rho_j^k}\rightarrow G(\zeta-\alpha)=\bar{G}(\zeta)},$$
where the zeros and poles of $\bar{G}$ are of multiplicity at least $n$, except the pole at $\zeta=0$, which has order $l$,  since  $\displaystyle{\frac{g_j(\rho_j\zeta)}{\rho_j^k}= \frac{f_j(\rho_j\zeta)}{\psi_j(\rho_j\zeta)\rho_j^k}=\frac{f_j(\rho_j\zeta)}{\zeta^l \phi_j(\rho_j\zeta)\rho_j^{k+l}}},$ $f_j(\zeta)$ and $\psi_j(\zeta)$ do not  have common zeros and $\rho_j\rightarrow 0$ thus, for $j$ large enough, there exist $0<r<1$ such that $f_j(\rho_j\zeta)$ do not have zeros in $\Delta(0, r)$. Thus $\displaystyle{\frac{\rho_j^k}{g_j(\rho_j\zeta)}}$ is holomorphic in $\Delta(0, r)$ and $\zeta=0$ is the  only zero of ${\displaystyle\frac{\rho_j^k}{g_j(\rho_j\zeta)}}$. On the other hand, since $\bar{G}(\zeta)$ has a pole at $\zeta=0$, we have $\displaystyle{\frac{1}{\bar{G}(\zeta)}}$ has a zero at $\zeta=0$. Since zeros are isolated therefore for $0<r'<r$, there exists $\epsilon >0$ such that $\displaystyle{\left|\frac{1}{\bar{G}(\zeta)}\right|>\epsilon}$ whenever $|\zeta|=r'$ and $\displaystyle{\left|\frac{\rho_j^k}{g_j(\rho_j\zeta)}-\frac{1}{\bar{G}(\zeta)}\right|<\epsilon}$ when $j$ is large enough. Thus by Rouche's Theorem we conclude that $\displaystyle{\frac{1}{\bar{G}(\zeta)}}$ has a zero of multiplicity $l$ at $\zeta=0.$\\

Now, set \begin{equation}\label{eq4.1} H_j(\zeta)=\frac{f_j(\rho_j\zeta)}{\rho_j^{k+1}}.\end{equation} Then $$H_j(\zeta)=\frac{\psi_j(\rho_j\zeta)}{\rho_j^l}\frac{f_j(\rho_j\zeta)}{\rho_j^k\psi_j(\rho_j\zeta)}=\frac{\psi_j(\rho_j\zeta)}{\rho_j^l}\frac{g_j(\rho_j\zeta)}{\rho_j^k}.$$
Note that $\displaystyle{\lim_{j\rightarrow\ity}\frac{\psi_j(\rho_j\zeta)}{\rho_j^l}=\zeta^l}$, thus ${H_j(\zeta)\rightarrow \zeta^l\bar{G}(\zeta)=H(\zeta)}$ locally uniformly with respect to the spherical metric. Clearly, all zeros of $H(\zeta)$ have multiplicity at least $n$ and all non-zero poles of $H(\zeta)$ are of multiplicity at least $n$ and $H(0)\neq0$. From \eqref{eq4.1}, we get $$H_j^{(k)}(\zeta)-\frac{\psi(\rho_j\zeta)}{\rho_j^l}\rightarrow H^{(k)}(\zeta)-\zeta^l$$ locally uniformly on $\C$ minus set of poles of $\bar{G}$. Now, by Hurwitz's Theorem all zeros of $H^{(k)}(\zeta)-\zeta^l$ has multiplicity  at least $m$ this shows that $H(\zeta)$ is non-constant otherwise  $H^{(k)}(\zeta)-\zeta^l=-\zeta^l$. Now, if $H(\zeta)$ is a transcendental function, by Logarithmic Derivative Theorem and Nevanlinna's first fundamental theorem, we obtain
\begin{align*}
m\left(r, \frac{1}{H}\right)&+ m\left(r, \frac{1}{H^{(k)}-\zeta^l}\right)\\
&\leq m\left(r, \frac{1}{H^{(k+1)}}\right) + m\left(r, \frac{1}{H^{(k+l)}-l!}\right) + S(r, H)\\
&=m\left(r, \frac{1}{H}+\frac{1}{H^{(k)}-\zeta^l}\right) + S(r, H)\\
&\leq m\left(r, \frac{1}{H^{(k+l+1)}}\right)+ S(r, H)\\
&\leq T\left(r, H^{(k+l+1)}\right)- N\left(r, \frac{1}{H^{(k+l+1)}}\right)\\
&\leq T\left(r, H^{(k)}\right) + (l+1)\overline{N}(r, H) - N\left(r, \frac{1}{H^{(k+l+1)}}\right) + S(r, H).
\end{align*}
Then we obtain
\begin{align*}
T(r, H)&\leq (l+1)\overline{N}(r, H)+ N\left(r, \frac{1}{H}\right)+ N\left(r, \frac{1}{H^{(k)}-\zeta^l}\right) - N\left(r, \frac{1}{H^{(k+l+1)}}\right) + S(r, H)\\
&\leq \frac{l+1}{n} N(r, H)+ (k+l+1)\overline{N}\left(r, \frac{1}{H}\right)+ (l+1)\overline{N}\left(r, \frac{1}{H^{(k)}-\zeta^l}\right) + S(r, H)\\
&\leq \left(\frac{l+1}{n}+\frac{k+l+1}{n} \right)T(r, H) + \frac{l+1}{m}\left(T(r, H)+k \overline{N}(r, H) \right)+ S(r, H)\\
&\leq \left(\frac{k+2(l+1)}{n}+\frac{l+1}{m} +\frac{k(l+1)}{mn} \right)T(r, H) + S(r, H).
\end{align*}
This is a contradiction since $\displaystyle{\frac{k+2(p+1)}{n}+\frac{p+1}{m} +\frac{k(p+1)}{mn}<1}$ and $l\leq p.$\\

If $H(\zeta)$ is a non-polynomial rational function. Then, set
\begin{equation}\label{eq4.2}
H(\zeta)=A\frac{(\zeta-\alpha_1)^{n_1}(\zeta-\alpha_2)^{n_2}\ldots (\zeta-\alpha_s)^{n_s}}{(\zeta-\beta_1)^{n_1'}(\zeta-\beta_2)^{n_2'}\ldots (\zeta-\beta_t)^{n_t'}},
\end{equation}
where $A$ is a non-zero constant and $n_1, n_1'\geq n$ are positive integers. Set $\displaystyle{\sum_{i=1}^{s}n_i = N\geq ns}$ and $\displaystyle{\sum_{j=1}^{t}n'_j = N'\geq nt}$. By argument principle and condition $(3)$ of the theorem $N'\geq N.$ From \eqref{eq4.2} we get
\begin{equation}\label{eq4.3}
H^{(k)}(\zeta)=\frac{(\zeta-\alpha_1)^{n_1-k}(\zeta-\alpha_2)^{n_2-k}\ldots (\zeta-\alpha_s)^{n_s-k}g_1(\zeta)}{(\zeta-\beta_1)^{n_1'+k}(\zeta-\beta_2)^{n_2'+k}\ldots (\zeta-\beta_t)^{n_t'+k}}=\frac{P(\zeta)}{Q(\zeta)},
\end{equation}
where $g_1(\zeta), P(\zeta) \ \text{and \ } Q(\zeta)$ are polynomials. By Lemma \ref{lemma 5}, deg $(g_1(\zeta))\leq k(s+t-1)$ and deg $\left(P(\zeta)\right)=N-ks+\text{deg\ }(g_1) $.  Again by \eqref{eq4.2} we have
\begin{align}
H^{(k+l+1)}(\zeta)&= \frac{(\zeta-\alpha_1)^{n_1-(k+l+1)}(\zeta-\alpha_2)^{n_2-(k+l+1)}\ldots (\zeta-\alpha_s)^{n_s-(k+l+1)}g_2(\zeta)}{(\zeta-\beta_1)^{n_1'+(k+l+1)}(\zeta-\beta_2)^{n_2'+(k+l+1)}\ldots (\zeta-\beta_t)^{n_t'+(k+l+1)}}\notag\\
&=\frac{P_1(\zeta)}{Q_1(\zeta)}\label{eq4.4}
\end{align}
where $g_2(\zeta), P_1(\zeta) \ \text{and\ } Q_1(\zeta)$ are polynomials such that $$\text{deg} (P_1)=N-(k+l+1)s+ \text{deg}\ (g_2).$$ By Lemma \ref{lemma 5} deg $(g_2(\zeta))\leq (k+l+1)(s+t-1).$ Again let
\begin{equation}\label{eq4.5}
H^{(k)}(\zeta) - \zeta^l = B\frac{(\zeta-\gamma_1)^{m_1}(\zeta-\gamma_2)^{m_2}\ldots (\zeta-\gamma_u)^{m_u}}{(\zeta-\beta_1)^{n_1'+k}(\zeta-\beta_2)^{n_2'+k}\ldots (\zeta-\beta_t)^{n_t'+k}}=\frac{P_2(\zeta)}{Q_2(\zeta)},
\end{equation}
where $B$ is a non-zero constant, $P_2(\zeta)$ and $Q_2(\zeta)$ are polynomials such that deg $(P_2)=M$  and $m_i\geq m$ are positive integers.  Set $\displaystyle{\sum_{i=1}^{u}m_i = M\geq mu}.$ From \eqref{eq4.5}, we get
\begin{align}
H^{(k+l+1)}(\zeta) &= \frac{(\zeta-\gamma_1)^{m_1-(l+1)}(\zeta-\gamma_2)^{m_2-(l+1)}\ldots (\zeta-\gamma_u)^{m_u-(l+1)}g_3(\zeta)}{(\zeta-\beta_1)^{n_1'+k+l+1}(\zeta-\beta_2)^{n_2'+k+l+1}\ldots (\zeta-\beta_t)^{n_t'+k+l+1}}\notag\\
&=\frac{P_3(\zeta)}{Q_3(\zeta)}=\frac{P_1(\zeta)}{Q_1(\zeta)},\label{eq4.6}
\end{align}
where $g_3(\zeta), P_3(\zeta)$ and $Q_3(\zeta)$ are polynomials such that $\text{deg\ }(P_3)= M-(l+1)u + \text{deg\ }(g_3)$ and by Lemma \ref{lemma 5}   deg $(g_3(\zeta))\leq (l+1)(u+t-1).$

Clearly, $\alpha_i\not=\gamma_j$ for $i=1,2, \ldots s$, $j=1,2, \ldots, t$, otherwise $H^{(k)}(\gamma_j)=0$, as zeros of  $H(\zeta)$ has multiplicity at least $m$ and by \eqref{eq4.5}, $\gamma_j=0$ and this shows that $H(0)=0$ which is  a contradiction to the fact that $H(0)\neq 0.$ Again from \eqref{eq4.2} and \eqref{eq4.5} we get
\begin{align}
H^{(k)}(\zeta)-\zeta^l&= \frac{(\zeta-\alpha_1)^{n_1-k}(\zeta-\alpha_2)^{n_2-k}\ldots (\zeta-\alpha_s)^{n_s-k}g_1(\zeta)}{(\zeta-\beta_1)^{n_1'+k}(\zeta-\beta_2)^{n_2'+k}\ldots (\zeta-\beta_t)^{n_t'+k}}-\zeta^l\notag\\
 &= B\frac{(\zeta-\gamma_1)^{m_1}(\zeta-\gamma_2)^{m_2}\ldots (\zeta-\gamma_u)^{m_u}}{(\zeta-\beta_1)^{n_1'+k}(\zeta-\beta_2)^{n_2'+k}\ldots (\zeta-\beta_t)^{n_t'+k}}.\label{eq4.7}
   \end{align}
From \eqref{eq4.7} we get
\begin{equation}\label{eq4.8}
M= \max\{N-ks+\text{deg\ }( g_1), \quad l+N'+kt\}
 \end{equation}
But
\begin{align}
N-ks +\text{deg \ }(g_1)&\leq N-ks +ks +kt-k\notag\\
&\leq N+kt-k\notag\\
&\leq N'+kt+l.\label{eq4.9}
\end{align}
 Thus from \eqref{eq4.8} and \eqref{eq4.9},  we get
  \begin{equation}\label{eq4.10}
   M=N'+kt+l
   \end{equation}

Since $\alpha_i\neq \gamma_j$ therefore from \eqref{eq4.4} and \eqref{eq4.6} $(\zeta-\gamma_i)^{m_i-(l+1)}$ is a factor of $g_2$. Thus we have
\begin{equation}\label{eq4.12}
M-(l+1)u\leq \text{deg\ }(g_2)\leq (k+l+1)(s+t-1).
\end{equation}

From \eqref{eq4.12} and \eqref{eq4.10} and noting that $\displaystyle{s\leq\frac{N}{n}}$, $\displaystyle{t\leq\frac{N'}{n}}$ and $\displaystyle{u\leq\frac{M}{m}}$ we obtain
\begin{align*}
\left( 1-\frac{l+1}{m}\right)(N'+kt+l)&\leq (k+l+1)\left(\frac{N}{n}+\frac{N'}{n}-1 \right)\\
&\leq (k+l+1)\left(\frac{2N'}{n}-1\right)
\end{align*}
and this shows that
\begin{equation}\notag
\left(1-\frac{l+1}{m}-2\left(\frac{k+l+1}{n}\right) \right)N'\leq -\left((k+l+1)+\left( 1-\frac{l+1}{m}\right)(l+kt)\right)<0,
\end{equation}
which is again a contradiction since $\displaystyle{\frac{p+1}{m}+2\left(\frac{k+p+1}{n}\right)<1}$ and $l\leq p$.\\

Now, we are left with the only case when $H(\zeta)$ is a non-constant polynomial. Set
\begin{equation}\label{eq4.13}
H(\zeta)= A(\zeta-\alpha_1)^{n_1}(\zeta-\alpha_2)^{n_2}\ldots (\zeta-\alpha_s)^{n_s},
\end{equation}
where $A$ is a non-zero constant and $n_i\geq n$ are positive integers. Set $\displaystyle{\sum_{i=1}^{s}n_i=N\geq ns}.$
Also set
\begin{equation}\label{eq4.14}
H^{(k)}(\zeta)-\zeta^l= B(\zeta-\beta_1)^{m_1}(\zeta-\beta_2)^{m_2}\ldots (\zeta-\beta_t)^{m_t},
\end{equation}
where $B$ is a non-zero constant and $m_j\geq m$ are positive integers. Set $\displaystyle{\sum_{j=1}^{t}m_j=M\geq mt}.$ From \eqref{eq4.13} and \eqref{eq4.14} we get $M=N-k.$\\

By the same reason as in the case of non-polynomial rational function we get $\alpha_i\neq \beta_j$ for $i=1, 2, \ldots, s,\ j=1, 2, \ldots t.$ Since zeros of $H(\zeta)$ and $H^{(k)}(\zeta)-\zeta^l$ are of multiplicity at least $n $ and $m$ respectively we observe that $\alpha_i$ and $\gamma_j$ are  zeros of $H^{(k+l+1)}$ with multiplicities at least $n_i-(k+l+1)$ and $m_j-(l+1)$ respectively. Therefore we get
\begin{equation}\notag
N-(k+l+1)s + M - (l+1)t\leq \text{deg\ }(H^{(k+l+1)})= N-(k+l+1).
\end{equation}
Since $M=N-k$ we get
\begin{equation}\notag
N\leq (k+l+1)s + (l+1)(t-1).
\end{equation}
Now, note that $\displaystyle{s\leq \frac{N}{n}}$ and $\displaystyle{t\leq\frac{M}{m}=\frac{N-k}{m}}.$ Therefore, we get
\begin{equation}\notag
\left( 1- \frac{k+l+1}{n}-\frac{l+1}{m} \right)N\leq -\left((l+1)\left(1+\frac{k}{m} \right) \right)<0.
\end{equation}
This is again a contradiction. Thus we have proved that $\mathcal G$ is normal in $\D.$\\

It remains to prove that $\fr$ is normal at $z=0$. Since $\mathcal{G}$  is normal at $0$ and $g_j(0)=\ity$ for all $j$, there exists $\delta > 0$ such that $|g_j(z)|\geq 1$ on $\Delta(0, \delta)$ for all $j$. Thus $f_j(z)\neq0$ on $\Delta(0, \delta)$, and hence  $1/f_j$ is holomorphic on $\Delta(0, \delta)$ for all $j.$ Choose $\delta $ small enough  that $\psi_j(z)\geq |z|^l/2$ for $|z|\leq \delta$ and $j$ sufficiently large, we have
\begin{equation}\notag
\left|\frac{1}{f_j^n(z)}\right|=\left|\frac{1}{g_j(z)}\frac{1}{\psi(z)} \right|\leq \frac{2^{l+1}}{\delta^m} \ \text{\ for\ }|z|=\frac{\delta}{2}.
\end{equation}
By the maximum modulus principle, this holds throughout $\Delta(0, \delta/2).$ This shows that $\fr$ is normal. \end{proof}

\textbf{Acknowledgements.} \quad  We wish to thank  Indrajit Lahiri (Kalyani University) and  Kaushal Verma (IISc Bangalore) for their  valuable suggestions and helps.

\end{document}